%% file: table_agosto_sewing_piecewise_space.tex
\documentclass[12pt,a4paper,reqno]{amsart}

\usepackage{amsmath}
\usepackage{amsfonts}
\usepackage{enumerate}
\usepackage{overpic}
\usepackage[margin=3cm]{geometry}
\usepackage{color}				
\usepackage[section]{placeins}
\usepackage{float}

\newcommand{\R}{\ensuremath{\mathbb{R}}}
\newcommand{\C}{\ensuremath{\mathbb{C}}}
\newcommand{\wx}{\widetilde{x}}
\newcommand{\ot}{\overline{\tau}}
\newcommand{\wy}{\widetilde{y}}
\newcommand{\wB}{\overline{B}}
\newcommand{\wC}{\overline{C}}

\DeclareMathOperator{\re}{Re}
\DeclareMathOperator{\im}{Im}

\DeclareMathOperator{\sgn}{sgn}

\DeclareMathOperator{\e}{e}
\DeclareMathOperator{\x}{x}
\DeclareMathOperator{\Sa}{Sa}
\DeclareMathOperator{\No}{No}
\DeclareMathOperator{\Nd}{Nd}
\DeclareMathOperator{\Fo}{Fo}
\DeclareMathOperator{\Ce}{Ce}

\DeclareMathOperator{\DD}{D}

\newtheorem{thm}{Theorem}

\newtheorem{cor}[thm]{Corollary}
\newtheorem{lema}[thm]{Lemma}
\newtheorem{prop}[thm]{Proposition}
\newtheorem{defi}[thm]{Definition}

\newtheorem{ex}[thm]{Example}

\definecolor{blue}{RGB}{41,5,195}

\title[Limit cycles and invariant surfaces of sewing PWL diff. systems on   $\mathbb{R}^3$]
{On the existence of limit cycles and invariant surfaces of sewing piecewise linear differential systems on   $\mathbb{R}^3$}

\author[B.R. De Freitas]{Bruno R. de Freitas}
\address{Instituto de Matem\'atica e Estat{\'\i}stica, Universidade Federal de
Goi\'as, 74001-970 Goi\^ania, Goi\'as, Brazil}
\email{freitasmat@ufg.br}

\author[J.~C.~Medrado]{Jo\~ao C. Medrado}
\address{Instituto de Matem\'atica e Estat{\'\i}stica, Universidade Federal de
Goi\'as, 74001-970 Goi\^ania, Goi\'as, Brazil}
\email{medrado@ufg.br}



\begin{document}

\begin{abstract}
We consider  a class of discontinuous piecewise linear differential systems in $\mathbb{R}^3$ with two pieces separated by a plane. In this class we show that there exist differential systems having: $(i)$  a unique limit cycle, $(ii)$ a unique one-parameter family of periodic orbits, $(iii)$ scrolls, $(iv)$ invariant cylinders foliated by orbits which can be periodic or no. 
\end{abstract}

\maketitle

\section{Introduction and Main Results }

The discontinuous piecewise linear differential systems plays an important role inside the theory of nonlinear dynamical systems.   In the models of physical problems or processes is natural to use  the piecewise-smooth dynamical systems when their motion is characterized by smooth flow and eventually, interrupted by instantaneous events (see~\cite{Brogliato}, \cite{Gou2}, \cite{Gou1}). There are many non-smooth processes in this context, for example, impact, switching, sliding and other discrete state transitions.  They are used also in nonlinear engineering models, where certain devices are accurately modeled by them, see for instance \cite{BBCK},  \cite{ML}, \cite{Si} and, references quoted
in these.

\smallskip

There are several papers about the existence of limit cycles for discontinuous piecewise linear differential systems on the plane. The study of this class of differential systems on $\R^3$  is beginning. In \cite{CAR}, \cite{PON2} and \cite{PON1},  the authors consider a family of continuous piecewise linear systems in $\mathbb{R}^3$ and characterize  limit cycles and cones foliated by periodic orbits. In~\cite{FIR} is proved the existence of  limit cycles and invariant cylinders for a class of discontinuous vector field in dimension $2n$.  

\smallskip

In this work, we consider  a class of discontinuous piecewise linear differential systems in $\mathbb{R}^3$ with two pieces separated by a plane and we investigate the existence of limit cycles and invariant surfaces. In this way,  we give conditions for the existence of differential systems having: $(i)$  a unique limit cycle (Figure~\ref{1limitcycle}); $(ii)$ a unique one-parameter family of periodic orbits (Figure~\ref{cone}); $(iii)$ scrolls (Figure~\ref{scroll}.$(a)$); $(iv)$ invariant cylinders foliated by orbits periodic or no (Figure~\ref{scroll}.$(b)$ and \ref{scroll}.$(c)$).  

\smallskip

We note that the existence of the one-parameter family of periodic orbits is an analogous result to the Lyapunov Center Theorem related to smooth vector fields. See, for instance  \cite{LCHR}, \cite{LCR}, \cite{LCR4D} and, \cite{LCH}. 

\smallskip

This work is an extension of~\cite{Medrado} and in our approach we use essentially the Theorem of Rolle for dynamical systems (see \cite{Rolle}) to address  the problem of to show the existence of limit cycles or invariant surfaces to find zeroes of intersection of algebraic curves.

\smallskip

Let $X^\pm(\x)=A^\pm (\x)+B^\pm$ be polynomial vector fields on $\mathbb{R}^3$ of degree one where $\x=(x,y,z)$, $A\pm=(a_{ij}^\pm)_{3\times 3}$, $B=(b_1^\pm,b_2^\pm,b_3^\pm)$ and $h$ the real function given by $h(x,y,z)=z$. We consider the piecewise linear vector field $Z=(X^+,X^-)$ defined by
\begin{equation}\label{sistemainicial}
Z(x,y,z)=\left\{\begin{array}{rc}
X^+(x,y,z),&\mbox{if}\quad (x,y,z)\in \Sigma^+,\\
X^-(x,y,z), &\mbox{if}\quad (x,y,z)\in\Sigma^-,
\end{array}\right.
\end{equation}
where  $\Sigma=h^{-1}(0)$ and $\Sigma^{\pm}=\pm h>0$. We observe that $\mathbb{R}^3=\Sigma\cup\Sigma^+\cup \Sigma^-$.

On $\Sigma^+ (\Sigma^-)$ the orbit $\varphi_Z(t,p)$ of the piecewise vector field $Z$ is given by the orbit $\varphi_X^+(t,p) (\varphi_X^-(t,p))$  of $X^+ (X^-)$, respectively. At points of  $\Sigma$ the orbit of $Z$ is bivalued on the \emph{sewing region} given by $\Sigma_{S}=\{p\in\Sigma;X^+h(p)X^-h(p)>0\}$ or in the \emph{tangency set} $\Sigma_T= \{p\in\Sigma;X^+h(p)=0 \mbox{ or } X^-h(p)=0\}$. If a point $p$ belongs to the $\Sigma_S$ or  $\Sigma_T$, the orbit of $Z$ through this point is given by the union of the orbits of  $X^+$ and $X^-$ through this same point. 

We say that an orbit $\gamma$ of vector field $Z=(X^+,X^-)$ is a $\Sigma-$\textit{sewing periodic orbit} or simply, $\Sigma-$\textit{periodic orbit},  if $\gamma$ is closed and $\gamma\cap \Sigma_S\neq \emptyset$. In $\Sigma\setminus (\Sigma_S\cup\Sigma_T)$  the definition of the orbit of $Z$ is given following the convetion of Filippov in \cite{Fil} .



The tangency set $\Sigma_T$ of $Z$  is formed by the \textit{tangency straight lines} $L_{X^+}$ and $L_{X^-}$,  i.e., $\Sigma_T=L_{X^+}\cup L_{X^-}$, where $L_{X^\pm}=\{ p\in\Sigma: X^\pm h(p)=0\}$.  If $p\in \Sigma_T$ it can be of \textit{fold}  or \textit{cusp} of piecewise vector field $Z$. In the first case we say that $p$ is a \textit{visible (invisible) fold} of $Z$ if $\sgn((\pm X^\pm)^2h(p)>0$  ($<0$), respectively. If $p\in \Sigma_T$ is not a fold of $Z$ but $(X^+)^3h(p)\neq 0$ or  $(X^-)^3h(p)\neq 0$ then $p$ is a cusp of $Z$.


%

In this work, we deal with piecewise linear vector field $Z=(X^+,X^-)$ expressed by \eqref{sistemainicial} with $\Sigma=\Sigma_S \cup \Sigma_T^2$, where  \[ \Sigma_T^2\subset\Sigma_T =\{ p\in \Sigma_T: p \mbox{ is an invisible fold point of }\\ X^+ \mbox{ or } X^-\}.\] So, we get the canonical form for to the vector field $Z=(X^+,X^-)$   expressed by
\begin{equation}\label{formanormal}
\begin{array}{lll}
X^+=(a^+x+b^+z,c^+y+d^+z-1,y),\\
X^-=(a^-x+b^-z+m,c^-y+d^-z+1,y).
\end{array}
\end{equation}
 
 The eigenvalues of $X^+$ and $X^-$ are 
 \[ \lambda_1^\pm=a^\pm, \; \lambda_{2}^\pm=(c^\pm+\sqrt{(c^\pm)^2+4d^\pm})/2 \mbox{ and, }\lambda_{3}^\pm=(c^\pm-\sqrt{(c^\pm)^2+4d^\pm}/2. 
 \]
 
 We observe that $Z=(X^+, X^-)$ has distinct dynamics depending of values of these eigenvalues. In order to analyze all possible dynamic types of $X^\pm$ and become easier the presentation of our results, we define  seven types:
 \begin{description} 
\item[$(i)$ $\Sa$] If $\lambda_2^\pm \lambda_3^\pm<0$.
\item[$(ii)$ $\No$] If $\lambda_2^\pm, \lambda_3^\pm\in\R$ and $\lambda_2^\pm \lambda_3^\pm>0$. 
\item[$(iii)$ $\Nd$] If $\lambda_2^\pm =\lambda_3^\pm$.
\item[$(iv)$ $\Fo$] If $\lambda_2^\pm, \lambda_3^\pm\in\C$ and $\re(\lambda_2^\pm)\im(\lambda_2^\pm)\neq0$.
\item[$(v)$ $\Ce$] If $\lambda_2^\pm, \lambda_3^\pm\in\C$, $\re(\lambda_2^\pm)=0$ and, $\im(\lambda_2^\pm)\neq0$.
\item[$(vi)$ $\DD_1$] If $\lambda_2^\pm \lambda_3^\pm=0$ and $(\lambda_2^\pm)^2 +(\lambda_3^\pm)^2\neq 0$.
\item[$(vii)$ $\DD_2$] If $(\lambda_2^\pm)^2 +(\lambda_3^\pm)^2= 0$. 
 \end{description}
 
 We remark that in the types $\DD_1$ and $\DD_2$, $X^\pm$ does not have equilibrium points. 


\begin{defi}
We say that the piecewise vector field $Z=(X^+,X^-)$ is of  type $(T^+, T^-)$ for  $T^\pm\in\{\mbox{$\Sa$, $\No$, $\Nd$, $\Fo$, $\Ce$, $\DD_1$, $\DD_2$}\}$, if  $X^\pm$ is of type $T^\pm$. 
\end{defi}
%
%
We observe that the type $(T^+, T^-)$ is equal to $(T^-, T^+)$ i.e., there is an equivalence between $(T^+, T^-)$ and $(T^-, T^+)$, for details see \cite{Torres}.

 In this paper we prove the following main theorems.
 
 \begin{thm}\label{teoprincipal}
Let $Z=(X^+,X^-)$ be a piecewise linear vector field of type $(T^+, T^-)$.
The following statements hold.
\begin{enumerate}
\item The vector field $Z$ has scrolls (see Figure~\ref{scroll}$(a)$) if only if is true one of the following conditions:  
\begin{enumerate}
\item $T^+=\mbox{$\Sa$}$ and $T^-\in\{\mbox{$\Sa$, $\No$, $\Nd$, $\Fo$, $\Ce$, $\DD_1$}\}$;
\item $T^+=\mbox{$\No$}$ and $T^-\in\{\mbox{$\No$, $\Nd$, $\Fo$, $\Ce$, $\DD_1$, $\DD_2$}\}$;
\item $T^+=\mbox{$\Nd$}$ and $T^-\in\{\mbox{$\Nd$, $\Fo$, $\DD_1$, $\DD_2$}\}$;
\item $T^+=\mbox{$\Fo$}$ and $T^-= \mbox{$\DD_1$} $;
\item $T^+=\mbox{$\DD_1$}$ and $T^-\in\{\mbox{$\DD_1$,$\DD_2$}\}$;
\end{enumerate}
with $\kappa^2+\lambda^2\neq0$ and $\kappa\lambda\geq0$, or $1+\alpha^2\lambda/\kappa\le0$ and $\kappa\lambda<0$.

\item The vector field $Z$ has at most a unique invariant cylinder (see Figure~\ref{scroll}$(b)$)  if only if is true one of the following conditions:  
\begin{enumerate}
\item $T^+=\mbox{$\Sa$}$ and $T^-\in\{\mbox{$\Sa $, $\No$, $\Nd$, $\Fo$, $\Ce$, $\DD_1$, $\DD_2$}\}$;
\item $T^+=\mbox{$\No$}$ and $T^-\in\{\mbox{$\No$, $\Nd$, $\Fo$, $\Ce$, $\DD_1$}\}$;
\item $T^+=\mbox{$\Nd$}$ and $T^-\in\{\mbox{$\Nd$, $\Fo$, $\Ce$, $\DD_1$}\}$;
\item $T^+=\mbox{$\Fo$}$ and $T^-\in\{\mbox{$\Fo$, $\Ce$, $\DD_1$, $\DD_2$}\}$;
\item $T^+=\mbox{$\Ce$}$ and $T^-=\mbox{$\DD_1$}$;
\end{enumerate}
with $\kappa\lambda<0$ and $1+\alpha^2\lambda/\kappa>0$.

\item The vector field $Z$ has infinitely many invariant cylinders (see Figure~\ref{scroll}$(c)$)  if only if  $\kappa=\lambda=0$ and is true one of the following conditions:
\begin{enumerate}
\item $T^+=\mbox{$\Sa$}$ and $T^-\in\{\mbox{$\Sa$, $\No$, $\Nd$, $\Fo$, $\Ce$, $\DD_1$}\}$;
\item $T^+=\mbox{$\No$}$ and $T^-\in\{\mbox{$\No$, $\Nd$, $\Fo$, $\Ce$, $\DD_1$}\}$;
\item $T^+=\mbox{$\Nd$}$ and $T^-\in\{\mbox{$\Nd$, $\Fo$, $\DD_1$}\}$;
\item $T^+=\mbox{$\Ce$}$ and $T^-\in \{\mbox{$\Ce$, $\DD_2$}\} $;
\item $T^+=\mbox{$\DD_1$}$ and $T^-=\mbox{$\DD_1$}$;
\item $T^+=\mbox{$\DD_2$}$ and $T^-=\mbox{$\DD_2$}$.
\end{enumerate}
\end{enumerate}
The parameters $\kappa$ and $\lambda$ depend on the parameters $a^\pm, b^\pm, c^\pm, d^\pm$ and $m$ of $Z$. These parameters are given in the Tables~\ref{table1} and~\ref{table3}.
\end{thm}
{\renewcommand{\arraystretch}{1.1}
\begin{table}[!h]
\caption{We present the expressions of  $\kappa$ and $\lambda$ that depend of parameters of the vector field $Z$. We note that $\alpha$ and $\beta$ are given in Lemmas~\ref{selacentrosup} and~\ref{selacentroinf}, respectively.}
\begin{center}
\begin{tabular}{c|l}
\hline
 (\mbox{$T^+$,$T^-$}) &  \hskip4cm $\kappa$ and $\lambda$\\
\hline
 (\mbox{$\Sa$,$\Sa$}) & $\begin{array}{ll}
                                         \kappa=&\alpha^2(\alpha c^-+\beta c^+-c^+-c^-)(\alpha\beta c^--\beta c^+-\beta c^-+c^+), \\
                                          \lambda=&(\alpha\beta c^+-\alpha c^+-\alpha c^-+c^-)(\alpha\beta c^++\alpha\beta c^--\alpha c^+-\beta c^-).

                                        \end{array}$ \\  
\hline  
   (\mbox{$\Sa$,$\No$}) & $\begin{array}{ll}
                                          \kappa=& -(\beta c^+-c^++(-1+\alpha)c^-)((c^++(-\alpha+1)c^-)\beta-c^+)\alpha^2, \\
                                          \lambda = & (((c^++c^-)\beta-c^+)\alpha-\beta c^-)((\beta c^+-c^+-c^-)\alpha+c^-).
                                        \end{array}$ \\  
\hline  
   (\mbox{$\Sa$,$\Nd$}) & $\begin{array}{ll}
                                          \kappa=& \alpha^2((-\alpha+1)\beta+c^+)^2, \\
                                          \lambda = & -((\beta+c^+)\alpha-\beta)^2.
                                        \end{array}$ \\  
\hline
(\mbox{$\Sa$,$\Fo$}) & $\begin{array}{ll}
                                          \kappa=& 4\alpha^2((\beta^2+1/4)(-1+\alpha)^2(c^-)^2-c^+(-1+\alpha)c^-+(c^+)^2),\\
                                          \lambda = & -(4(\beta^2+1/4))(-1+\alpha)^2(c^-)^2-4c^+\alpha(-1+\alpha)c^--4(c^+)^2\alpha^2.
                                        \end{array}$ \\ 
\hline
(\mbox{$\Sa$,$\Ce$}) & $\begin{array}{ll}
                                          \kappa=& ((-1+\alpha)^2\beta^2+(c^+)^2)\alpha^2, \\
                                          \lambda = & -(-1+\alpha)^2\beta^2-(c^+)^2\alpha^2.
                                        \end{array}$ \\ 
\hline 
(\mbox{$\Sa$,$\DD_1$}) & $\begin{array}{ll}
                                       \kappa=& -\alpha((-\alpha+1)c^-+c^+), \\
                                       \lambda = &(c^++c^-)\alpha-c^-.
                                       \end{array}$ \\
\hline
 (\mbox{$\No$,$\No$}) & $\begin{array}{ll}
                                          \kappa=& -(\beta c^+-c^++(-\alpha-1)c^-)((c^++(1+\alpha)c^-)\beta-c^+)\alpha^2, \\
                                          \lambda = & ((\beta c^+-c^+-c^-)\alpha-c^-)(((c^++c^-)\beta-c^+)\alpha+\beta c^-).
                                        \end{array}$ \\  
\hline
(\mbox{$\No$,$\Nd$}) & $\begin{array}{ll}
                                          \kappa=& ((1+\alpha)\beta+c^+)^2\alpha^2,\\
                                          \lambda = & -((\beta+c^+)\alpha+\beta)^2.
                                        \end{array}$ \\  
          \hline
          (\mbox{$\No$,$\Fo$}) & $\begin{array}{ll}
                                                    \kappa=& (4((1+\alpha)^2(\beta^2+1/4)(c^-)^2+c^+(1+\alpha)c^-+(c^+)^2))\alpha^2,\\
                                                    \lambda =& -4(1+\alpha)^2(\beta^2+1/4)(c^-)^2-4c^+\alpha(1+\alpha)c^--4(c^+)^2\alpha^2 .
                                                  \end{array}$ \\  
                                                                          
          \hline
          (\mbox{$\No$,$\Ce$}) & $\begin{array}{ll}
                                                    \kappa=&\alpha^2((1+\alpha)^2\beta^2+(c^+)^2), \\
                                                    \lambda =& (-\beta^2-c^+)^2)\alpha^2-2\alpha\beta^2-\beta^2.
                                                  \end{array}$ \\                               
       \hline
       (\mbox{$\No$,$\DD_1$}) & $\begin{array}{ll}
                                                 \kappa=&\alpha((1+\alpha)c^-+c^+), \\
                                                 \lambda =& (-c^+-c^-)\alpha-c^-.
                                               \end{array}$ \\  
                                                                       
       \hline   
(\mbox{$\Nd$,$\Nd$}) & $\begin{array}{ll}
                                          \kappa=& -(\alpha+\beta),\\
                                          \lambda = & 2\beta.
                                        \end{array}$ \\  
                                                                
\hline
(\mbox{$\Nd$,$\Fo$}) & $\begin{array}{ll}
                                          \kappa=& -(4\beta^2(c^-)^2+4\alpha^2+4\alpha c^-+(c^-)^2),\\
                                          \lambda =& (8\beta^2+2)(c^-)^2+4\alpha c^-.
                                        \end{array}$ \\  
\hline
(\mbox{$\Nd$,$\DD_1$}) & $\begin{array}{ll}
                                          \kappa=&-(\alpha\beta+1),\\
                                          \lambda =&1.
                                        \end{array}$ \\                                                              
\hline
(\mbox{$\DD_1$,$\Fo$}) &    $\begin{array}{l}  \kappa=4(1+c^-\alpha)+(c^-)^2\alpha^2(4\beta^2+1),\\  \lambda=-(c^-)^2\alpha^2(4\beta^2+1).\end{array}$\\ 
\hline
(\mbox{$\DD_1$,$\DD_1$}) &    $\kappa=\lambda=\alpha+\beta.$\\
\hline                   
\end{tabular}
\end{center}
\label{table1}
\end{table}                                                          

\begin{table}[!h]
 \caption{In this table are given $\kappa$ and $\lambda$ for the cases that they are constant.}
 \begin{center}
 \begin{tabular}{l|l} 
\hline
 \hskip4cm (\mbox{$T^+$,$T^-$}) &  $ \hskip1cm  \kappa$ and $\lambda$\\
\hline
(\mbox{$\Sa$,$\DD_2$}), (\mbox{$\Ce$,$\Nd$}),  (\mbox{$\Ce$,$\Fo$}),  (\mbox{$\Ce$,$\DD_1$}) \mbox{ and } (\mbox{$\DD_2$,Fo})   &   $\kappa=1$ \mbox{ and } $\lambda=-1$.    \\ 
  \hline
(\mbox{$\No$,$\DD_2$}), (\mbox{$\Nd$,$\DD_2$})  and  (\mbox{$\DD_1$,$\DD_2$})  &   $\kappa=\lambda=1$.  \\
\hline
(\mbox{$\Ce$,$\Ce$}), (\mbox{$\DD_2$,$\Ce$}) and (\mbox{$\DD_2$,$\DD_2$}) &   $\kappa=\lambda=0$. \\
\hline
  \end{tabular}
  \end{center}
  \label{table3}
  \end{table} 
{\renewcommand{\arraystretch}{1}  
 
\begin{figure}[!htb]
\begin{center}
\def\svgwidth{0.75\textwidth}
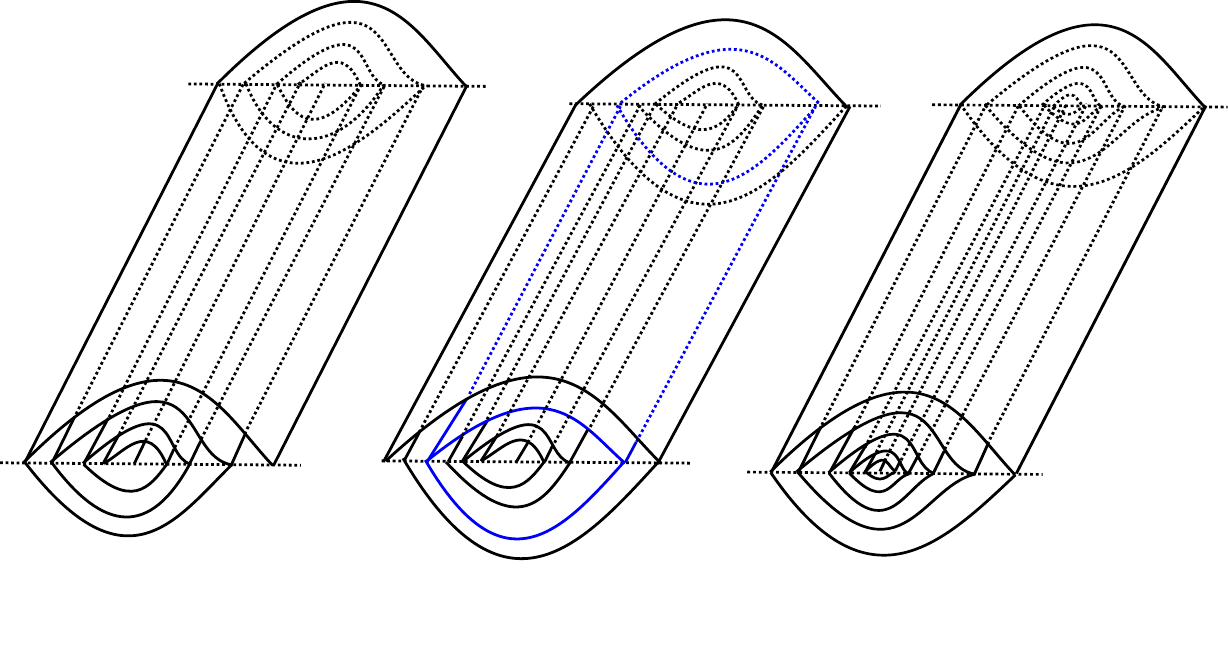
\end{center}
\caption{Invariant surfaces of \eqref{formanormal}: $(a)$ Scroll. $(b)$ A unique invariant cylinder.  $(c)$ Infinitely many invariant cylinders.}\label{scroll}
\end{figure}
%
%
%

\begin{thm}\label{teoprincipal2}
Let $Z=(X^+,X^-)$ be a piecewise linear vector field with $X^+$ and $X^-$ defined in~\eqref{formanormal}. The following statements hold.
\begin{enumerate}
\item If $(a^+)^2+(a^-)^2=0$ or $Z$ has no invariant cylinder then there are not limit cycles. 

\item If $(a^+)^2+(a^-)^2\neq0$ and $Z$ has at most a unique invariant cylinder then $Z$ has at most a unique limit cycle in this cylinder.

\item If $(a^+)^2+(a^-)^2\neq0$, $a^+a^-\geq0$ and $Z$ has infinitely many invariant cylinders then there is an invariant surface formed of periodic orbits, where each periodic orbit is contained in an invariant cylinder.

\end{enumerate}

\end{thm}

%
%
%
%
%
%
The statement~$(3)$ of Theorem~\ref{teoprincipal2} is similar a Lyapunov Center Theorem related to smooth vector fields. 


In the Section~\ref{formacanonica} we obtain the  canonical form~\eqref{formacanonica} and characterize the tangencial sets (straight lines of tangency) formed by orbits of $Z$ and $\Sigma$. The Theorems~\ref{teoprincipal} and~\ref{teoprincipal2}  will be proved in Section~\ref{provateoprincipal}.

\section{The tangency straight lines and canonical forms of $Z$}\label{formacanonica}
The purpose of this section is to characterize the tangency straight lines $L_{X^\pm}$ and to obtain a canonical form to the piecewise vector field~\eqref{sistemainicial} such that $\Sigma=\Sigma_S \cup \Sigma_T^2$ with $\Sigma_T^2=L_{X^+}^2 \cup L_{X^-}^2$, where $L_{X^\pm}^2 = \{p\in \Sigma: X^\pm h(p)=0 \mbox{ and } (\pm X^\pm)^2h(p)<0 \}$  and the sets $\Sigma_T, L_{X^+}$ and $L_{X^-}$ are not empty.   As $L_{X^+}$ and $L_{X^-}$ are not empty  we get that  $(a_{31}^\pm)^2+(a_{32}^\pm)^2\neq0$. In the sequel, we will assume these hypothesis.

\begin{lema}\label{retasiguais}
Let $Z=(X^+,X^-)$ be defined in~\eqref{sistemainicial} with $L_{X^\pm}$ the tangency straight lines of $X^\pm$. If  $X^+h(p)X^-h(p)\ge0$, for all $p\in\Sigma$ then the tangency straight lines are the same, i.e., $L_{X^+}\equiv L_{X^-}$.
\end{lema}
\begin{proof}
Each tangency straight line $L_{X^+}$ or $L_{X^-}$ separate $\Sigma$ in two regions where the signals of $X^\pm(p)$ are opposite. So, by the hypothesis we get $L_{X^+}\equiv L_{X^-}$ and without loss of generality we can assume that  the tangency straight line contains the origin. Moreover, we get that $L_X^\pm$ in $\Sigma$ is given by $y=0$.

\end{proof}
 Using the Lemma~\ref{retasiguais} and  rescaling the time, we have that the vector fields $X^+$ and $X^-$ can be written by
 \begin{equation}
 \label{quasenormal}
 \begin{array}{l}
 X^+(x,y,z)=(a_{11}^+x+ a_{12}^+y+a_{13}^+z+b_1^+,a_{21}^+x+a_{22}^+y+a_{23}^+z+b_2^+,y+a_{33}^+z),\\
 X^-(x,y,z)=(a_{11}^-x+ a_{12}^-y+a_{13}^-z+b_1^-,a_{21}^-x+a_{22}^-y+a_{23}^-z+b_2^-,y+a_{33}^-z).
 \end{array}
 \end{equation}
 

\begin{lema}\label{confreta}
Let $X^+$ be as given in~\eqref{quasenormal}. The tangency straight line $L_{X^+}$ is characterized by one of the following statements:
\begin{enumerate}
\item If $a_{21}^+=0$ then for all $p\in L_{X^+}$, $p$ is a visible or  invisible fold when $b_2^+> 0$ or  $b_2^+<0$, respectively;

\item  If $a_{21}^+=b_2^+=0$ then $L_{X^+}$ is invariant by $X^+$;

\item If $a_{21}^+\neq 0$ then there is a unique point $p\in L_{X^+}$ such that  it is a cusp (if $a_{21}^+b_1^+-a_{11}^+b_2^+\neq 0$) or a singular point of $X^+$ (if $a_{21}^+b_1^+-a_{11}^+b_2^+= 0$).

\end{enumerate}
\end{lema}
 \begin{proof}
From \eqref{quasenormal} we get $X^+h(x,y,0)=y$ that implies $L_{X^+}=\{y=z=0\}$ and $(X^+)^2h(x,0,0)=a_{21}^+x+b_2^+$. If $a_{21}^+=0$ we get two possibilities: $(i)$  when $b_2^+\neq 0$, the point  $(x,0,0)$ is a fold for all $x\in \R$ (statement 1); $(ii)$ when $b_2^+=0$, we have that $L_{X^+}$ is invariant by the vector field $X^+$ (statement 2). Finally, for $a_{21}\neq 0$, follows from $X^2h(x,0,0)=0$ that $(X^+)^3h(-b_2^+/a_{21}^+,0,0)=-a_{11}^+b_2^++b_1^+a_{21}^+$.  Thus, if $-a_{11}^+b_2^++b_1^+a_{21}^+\neq 0$  the point $(-b_2^+/a_{21}^+,0,0)$ is a cusp, otherwise is a singular point.
 \end{proof}


\begin{prop}\label{propformacanonica}
Let $Z=(X^+,X^-)$ the vector field defined as~\eqref{sistemainicial}. Suppose that $X^+hX^-h(p)\ge0$ and that $L_{X^\pm}\equiv L_{X^\pm}^2$. Then the vectors field $X^+$ and $X^-$ can be written as~\eqref{formanormal}.
\end{prop}
\begin{proof}
From Lemmas~\ref{retasiguais} and~\ref{confreta} we get $L_{X^\pm}^2=\{y=z=0\}$, $a_{21}^\pm= 0$,  and that $\sgn(\pm b_2^\pm)<0$. 
Doing the twin change of variables given by\\
$\varphi^{+}(x,y,z)=(x-(b_1^+/b_2^+)y-((b_1^+a_{11}^+-b_1^+a_{22}^++a_{12}^+b_2^+)/b_2^+)z,y+a_{33}^+z,-b_2^+z)$ on $\Sigma^+$, \\
$\varphi^{-}(x,y,z)=(x-(b_1^+/b_2^+)y-((b_2^+a_{12}^--b_1^+a_{22}^-+b_1^+a_{11}^-)/b_2^+)z,y+a_{33}^-z,b_2^-z)$ on $\Sigma^-$,\\
 and rescaling the time $t\rightarrow-t/b_2^+$ on $\Sigma^+$ and $t\rightarrow,t/b_2^-$ on $\Sigma^-$ we obtain the canonical form~\eqref{formanormal}.
\end{proof}

}

\begin{prop}\label{sollinear}
Consider the boundary value problem  $\dot{\x}=Y^\pm(\x)=P\x+Q^\pm$ with $ \x_0=(x(0),y(0),z(0))=(x_0,y_0,0)$,
\[
P=\left(\begin{array}{ccc}
     \gamma & 0 & \delta \\
     0 & \sigma & \psi \\
     0 & 1 & 0
     \end{array}\right)
\mbox{ and }
Q^\pm =\left( \begin{array}{c}
               M\\ \pm 1 \\ 0
               \end{array}\right),
\]
where $\gamma, \delta, \sigma, \psi, M\in \R$.

Let $\varphi^\pm(t,(x_0,y_0,0))$ be the solutions of $\dot{\x}=Y^\pm(\x)$ and consider the straight line $r_0=\{(x,y,z)\in \R^3: y=y_0, z=0\}$. Let $\tau^\pm\in \R/\{0\}$ such that $z(\tau^\pm)=0$ then $\varphi^\pm(\tau^\pm,r_0)$ is a straight line parallel to $r_0$ given by $r_1=\{(x,y,z)\in \R^3: y=y_1, z=0\}$.
\end{prop}

\begin{proof}
The general solution $\varphi^\pm(t,(x_0,y_0,0))$ is 
\[\e^{Pt}\x_0+\e^{Pt}\int_0^t \e^{-P \eta}Qd\eta.\]

Observe that the matrix $\e^{Pt}$ has zeroes at positions $(2,1)$ and $ (3,1)$. So, we can write  the solution as $\varphi^\pm(t,(x_0,y_0,0))=(x^\pm(t),y^\pm(t),z^\pm(t))$ where

\begin{equation}\label{retas}
\begin{array}{ll}
x^\pm(t)= & \e^{\gamma t}  x_0+f_{12}^\pm y_0+f_{13}^\pm ,\\
y^\pm(t)= & f_{22}^\pm y_0+f_{23}^\pm ,\\
z^\pm(t)= & f_{32}^\pm y_0+f_{33}^\pm ,
\end{array}
\end{equation}
where $f_{ij}^\pm=f_{ij}^\pm(t, \gamma,\delta,\sigma,\psi,M), $ for $i,j=1,2,3$.

Now, as the Poincar\'e Application is well defined, there is a $\tau^\pm(y_0)$ such that $z^\pm(\tau^\pm(y_0))=0$. Then $y_1=y^\pm(\tau^\pm(y_0))$ depends only of $y_0$. This implies that all orbits of $Y^\pm$ with origin at $r_0$ intersect $\Sigma=\{z=0\}$ after time $\tau^\pm(y_0)$, i.e., $\varphi^\pm(\tau^\pm(y_0),r_0)$ is the straight line $r_1$.
\end{proof}

Follows from of Proposition~\ref{sollinear} that if $Z$ has a periodic orbit then it is in a invariant cylinder. So, all periodic orbit are $1-$periodic.

\begin{cor}\label{cilindros}
Consider the boundary value problems 
$$
(A):\left\{\begin{array}{lll}
\dot{\x}=X^+(\x),\\
(x(0),y(0),z(0))=(x_0,y_0,0),\\
(x(\tau),y(\tau),z(\tau))=(x_1,y_1,0),
 \end{array}\right.\ \ \ \ \ (B): \left\{\begin{array}{lll}
\dot{\x}= X^-(\x),\\
 (x(0),y(0),z(0))=(\wx_1,\wy_1,0),\\
 (x(\ot),y(\ot),z(\ot))=(\wx_0,\wy_0,0).
   \end{array}\right.
$$
where $X^+$ and $X^-$ are given by \eqref{formanormal}. If $y_0=\wy_0$ and $y_1=\wy_1$ then there is an invariant cylinder for the vector field $Z=(X^+,X^-)$.
\end{cor}
\begin{proof}
Let $\varphi^\pm(t,p)$ be the solutions of $(A)$ and $(B)$ respectively and the straight lines $r_0=\{(x,y,z)\in \R^3: y=y_0, z=0\}$ and $r_1=\{(x,y,z)\in \R^3: y=y_1, z=0\}$. From Proposition~\ref{sollinear}, we have that $\varphi^+(r_0,\tau)=r_1$ and $\varphi^-(r_1,\ot)=r_0$.  So, we obtain an invariant cylinder (See Figure~\ref{calhasupinf}).
\begin{figure}[!htb]
\begin{center}
\def\svgwidth{0.70\textwidth}
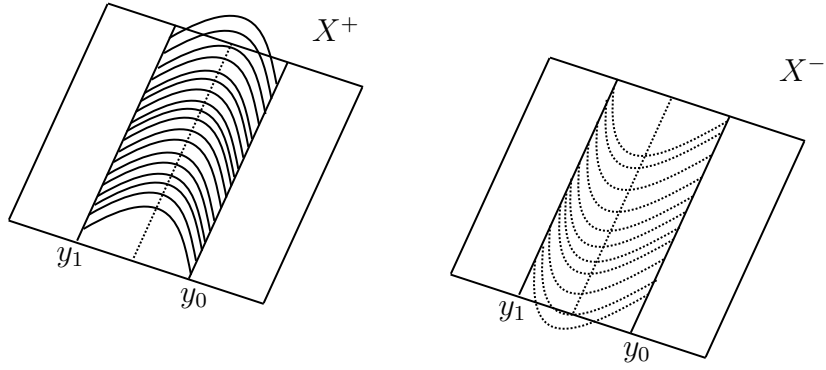
\end{center}
\caption{Superior half cylinder and inferior half cylinder.}\label{calhasupinf}
\end{figure}
\end{proof}
In the next section we will determine a maximum quota for the number of invariant cylinders and in the sequel a maximum quota for the number of cycles in each cylinder.

\smallskip

\section{Proof of Theorems }\label{provateoprincipal}
The proofs of Theorems~\ref{teoprincipal} and~\ref{teoprincipal2} is done in the subsections~\ref{secth1} and~\ref{secth2}, respectively,  except when $Z$ is of type ($\Fo$,$\Fo$) which is proved in the subsection~\ref{secth3}.  

In the subsections~\ref{secth1} and~\ref{secth2} we present a complete proof considering $Z$ of type  ($\Sa$,$\Sa$) under hypothesis given in the statements~(1), (2) and (3) of Theorems~\ref{teoprincipal} and~\ref{teoprincipal2} because for all other cases, except ($\Fo$,$\Fo$), the proof is done in similar way. We present in the Tables~\ref{table1} and~\ref{table3} the distinguished elements used in the proof. Finally, we conclude the proofs of these theorems  in the subsection~\ref{secth3} when  $Z$ is of the type ($\Fo$,$\Fo$).

In order to prove these theorems, when the return applications is defined,  we make a substitution of variables given by Propositions~\ref{selacentrosup} and~\ref{selacentroinf} and we address the proof to determine intersection points of curves which are associated to existence of invariant cylinders. For to determine the number of intersection points between these curves, we use also Theorem~\ref{rolle} proved by Kovanskii(\cite{Rolle}).

\begin{lema}\label{selacentrosup}
Consider the boundary value problem 
\begin{equation}\label{problemafronteira1}
\begin{array}{lll}
x'=a^+x+b^+z,\\
y'=c^+y+d^+z-1,\\
z'=y,
 \end{array}\ \ \ \ \  \begin{array}{lll}
 (x(0),y(0),z(0))=(x_0,y_0,0),\\
 (x(\tau),y(\tau),z(\tau))=(x_1,y_1,0),
   \end{array}
\end{equation}
where $a^+\neq0$. Then the next statements are true.
\begin{enumerate}
\item [$(i)$] If $X^+$ is of the type $\Sa$  doing $(\rho,v,w)=(\e^{a^+\tau},\e^{-(c^+ + s)\tau/2},\e^{(c^+ - s)\tau/2})$  then we get $v=\rho^{\alpha_1}, w=\rho^{\alpha_2}$ and $w=v^{\alpha}$, where $\alpha_1=-(c^+ + s)/2a^+$, $ \alpha_2=(c^+-s)/2a^+$, $s=\sqrt{(c^+)^2+4d^+}$, $\alpha=(s-c^+)/(s+c^+)$. Moreover, 
\begin{equation}\label{vwsa}
\begin{array}{c}
w=\dfrac{c^+y_1-sy_1-2}{c^+y_0-sy_0-2}, \ \ v=\dfrac{c^+y_0+sy_0-2}{c^+y_1+sy_1-2},\\
\; \\
\rho=\displaystyle\frac{4(a^+)^3x_1-4(a^+)^2c^+x_1+a^+(c^+)^2x_1-a^+s^2x_1+4a^+b^+y_1-4b^+}{4(a^+)^3x_0-4(a^+)^2c^+x_0+a^+(c^+)^2x_0-a^+s^2x_0+4a^+b^+y_0-4b^+}.
\end{array}
\end{equation}
\item [$(ii)$] If $X^+$ is of the type $\No$, doing $(\rho,v,w)=(\e^{a^+\tau},\e^{(c^++s)\tau/2},\e^{(c^+-s)\tau/2})$ then we get $v=\rho^{\alpha_1}, w=\rho^{\alpha_2}$ and $w=v^{\alpha}$, where  $\alpha_1=(c^++s)/2a^+$, $ \alpha_2=(c^+-s)/2a^+$, $s=\sqrt{(c^+)^2+4d^+}$, $\alpha=(c^+-s)/(s+c^+)$. Moreover,
$$
w=\frac{c^+y_1-sy_1-2}{c^+y_0-sy_0-2}, \ \ v=\frac{c^+y_1+sy_1-2}{c^+y_0+sy_0-2},
$$
$$
\rho=\displaystyle\frac{4(a^+)^3x_1-4(a^+)^2c^+x_1+a^+(c^+)^2x_1-a^+s^2x_1+4a^+b^+y_1-4b^+}{4(a^+)^3x_0-4(a^+)^2c^+x_0+a^+(c^+)^2x_0-a^+s^2x_0+4a^+b^+y_0-4b^+}.
$$

\item[$(iii)$] If $X^+$ is of the type $\Nd$, doing $(\rho,v,w)=(\e^{a^+\tau},\e^{c^+\tau/2},c^+\tau/2)$ then we get $v=\rho^{\alpha_1}, w=\ln(\rho^{\alpha_1})$ and $w=\ln(v)$, where  $\alpha_1=c^+/2a^+, \alpha=c^+/2$,
$$
w=\frac{2c^+(y_0-y_1)}{(c^+y_1-2)(cy_0-2)}, \ \ v=\frac{cy_1-2}{cy_0-2},
$$
$$
\rho=\frac{4(a^+)^3x_1-4(a^+)^2c^+x_1+a^+(c^+)^2x_1+4a^+b^+y_1-4b^+}{4(a^+)^3x_0-4(a^+)^2c^+x_0+a^+(c^+)^2x_0+4a^+b^+y_0-4b^+}.
$$

\item[$(iv)$] If $X^+$ is of the type $\Fo$, doing  $(\rho,v,w)=(\e^{a^+\tau},\e^{c^+\tau},\tan(s\tau/2))$ then we get $v=\rho^{\alpha_1}, w=\tan(\ln(\rho^{\alpha_2}))$ and $w=\tan(\ln(v^\alpha))$ where $\alpha_1=c^+/a^+,$ $\alpha_2=s/2a^+$, $s=\sqrt{-((c^+)^2+4d^+)}$, $\alpha=s/2c^+$,
$$
w=\frac{2s(-y_1+y_0)}{((c^+)^2+s^2)y_1-2c^+)y_0-2c^+y_1+4},\ v=\frac{4-4c^+y_1+((c^+)^2+s^2)y_1^2}{4-4c^+y_0+((c^+)^2+s^2)y_0^2}
$$
$$
\rho=\frac{4x_1(a^+)^3-4x_1(a^+)^2c^++(((c^+)^2+s^2)x_1+4b^+y+1)a-4b^+}{4(a^+)^3x_0-4c^+(a^+)^2x_0+(((c^+)^2+s^2)x_0+4b^+y_0)a^+-4b^+}.
$$

\item[$(v)$]If $X^+$ is of the type $\Ce$, doing $(\rho,v,w)=(\e^{a^+\tau},\cos(\alpha\tau),\sin(\alpha\tau))$ then we get $v=\cos(\ln(\rho^{\alpha/a^+})), w=\sin(\ln(\rho^{\alpha/a^+}))$ and $w^2+v^2=1$, where $\alpha=\sqrt{-d^+}$,
$$
w=\frac{\alpha(y_0-y_1)}{\alpha^2y_0^2+1}, \ \ v=\frac{\alpha^2y_0y_1+1}{\alpha^2y_0^2+1},
$$
$$
\rho=\displaystyle\frac{a^3x_1+a\alpha^2x_1+aby_1-b}{a^3x_0+a\alpha^2x_0+aby_0-b}.
$$
\item[$(vi)$] If $X^+$ is of the type $\DD_1$, doing $(\rho,v,w)=(\e^{a^+\tau},\e^{c^+\tau},c^+\tau)$ then we get $v=\rho^{\alpha_1}, w=\ln(\rho^{\alpha_1})$ and $w=\ln(v)$, where  $\alpha_1=c^+/a^+, \alpha=1/c^+$,
$$
w = y_0-y_1, \ \ v=\frac{c^+y_1-1}{c^+y_0-1},
$$
$$
\rho=\frac{(a^+)^3x_1-(a^+)^2c^+x_1+a^+b^+y_1-b^+}{(a^+)^3x_0-(a^+)^2c^+x_0+a^+b^+y_0-b^+}.
$$

\item[$(vii)$] If $X^+$ is of the type $\DD_2$, doing $(\rho,v,w)=(\e^{a^+\tau},\tau,\tau^2)$ then we get $v=\ln(\rho^{\alpha}), w=(\ln(\rho^{\alpha}))^2$ and $w=v^2$, where $ \alpha=1/a^+$,
$$
w=2y_0(y_0-y_1), \ \ v=y_0-y_1,
$$
$$
\rho=\frac{(a^+)^3x_1+a^+b^+y_1-b^+}{(a^+)^3x_0+a^+b^+y_0-b^+}.
$$

\end{enumerate}
\end{lema}

\begin{proof}
$(i)$ By the Proposition~\ref{sollinear} the solution $(x(t),y(t),z(t))$ is 
\[
(\e^{\lambda_1t}x_0+\eta y_0+\mu,
\dfrac{\lambda_2\e^{\lambda_2t}-\lambda_3\e^{\lambda_3t}}{s}y_0\pm\dfrac{\e^{\lambda_2t}-\e^{\lambda_3t}}{s},
\frac{\e^{\lambda_2t}-\e^{\lambda_3t}}{s}y_0\mp\frac{\e^{\lambda_3t}-\e^{\lambda_2t}-s}{\lambda_2\lambda_3}),
\]
where $\lambda_1=a^+$, $\lambda_{2,3}=(c^+\pm s)/2$ and $(\eta,\mu)$ obtained replacing $(\delta,M,\sigma,\psi,\pm1)=(b^+,0,c^+,d^+,-1)$. Now, we solve the system $\{x(\tau)=x_1, y(\tau)=y_1,z(\tau)=0\}$ in $\{\rho, v, w\}$ where $(\rho,v,w)=(\e^{a^+\tau},\e^{-(c^+ + s)\tau/2},\e^{(c^+ - s)\tau/2})$ and we obtain \eqref{vwsa}.\\
The remaining cases are obtained analogously.
\end{proof}

In similar way we obtain analogous result associated to $X^-$.

\begin{lema}\label{selacentroinf}
Consider the boundary value problem 
\begin{equation}\label{problemafronteira2}
 \begin{array}{lll}
x'=a^-x+b^-z+m,\\
y'=c^-y+d^-z+1,\\
z'=y,
 \end{array}\ \ \ \ \  \begin{array}{lll}
 (x(0),y(0),z(0))=(\wx_1,\wy_1,0),\\
 (x(\overline{\tau}),y(\overline{\tau}),z(\overline{\tau}))=(\wx_0,\widetilde{y_0},0), \\
   \end{array}
\end{equation}
where $a^-\neq0$. Then the next statements are true.
\begin{enumerate}
\item[$(i)$]  If $X^-$ is of the type $\Sa$ and $\No$,  doing $(\xi,V,W)=(\e^{a^-\overline{\tau}},\e^{-(c^-+S)\overline{\tau}/2},\e^{-(-c^-+S)\overline{\tau}/2})$ then we get $V=\xi^{\beta_1}, W=\xi^{\beta_2}$ and $W=V^\beta$, where $\beta_1=-(c^-+S)/2a^-$, $\beta_2=(c^--S)/2a^-, S=\sqrt{(c^-)^2+4d^-}$, $\beta=(S-c^-)/(S+c^-)$. Moreover,
\begin{equation}\label{VWsa}
\begin{array}{c}
W=\dfrac{S\wy_0-c^-\wy_0-2}{S\wy_1-c^-\wy_1-2}, \ \ V=\dfrac{S\wy_1+c^-\wy_1+2}{S\wy_0+c^-\wy_0+2},\\
\; \\
\xi=\displaystyle\frac{(S^2+4a^-(c^--a^-))m+(S^2+c^-(4a^--c^-)-4(a^-)^2)a^-\wx_0-4b^-(a^-\wy_0+1)}{(S^2+4a^-(c^--a^-))m+(S^2+c^-(4a^--c^-)-4(a^-)^2)a^-\wx_1-4b^-(a^-\wy_1+1)}.
\end{array}
\end{equation}

\item[$(ii)$]  If $X^-$ is of the type $\Nd$, doing $(\xi,V,W)=(\e^{a^-\overline{\tau}},\e^{c^-\ot/2},c^-\ot/2)$ then we get $V=\xi^{\beta_1}$, $W=\ln(\xi^{\beta_1})$ and $W=\ln(V)$, where $\beta_1=c^-/2a^-$, $\beta=c^-/2$,
$$
W=\dfrac{4c^-(\wy_0-\wy_1)}{2(c^-\wy_1+2)(c^-\wy_0+2)}, \ \ V=\dfrac{c^-\wy_0+2}{c^-\wy_1+2},
$$
$$
\xi=\frac{((c^-)^2-4c^-a^-+4(a^-)^2)m+((c^-)^2a^--4c^-(a^-)^2+4(a^-)^3)\wx_0+4b^-(a^-\wy_0+1)}{((c^-)^2-4c^-a^-+4(a^-)^2)m+((c^-)^2a^--4c^-(a^-)^2+4(a^-)^3)\wx_1+4b^-(a^-\wy_1+1)}.
$$

\item[$(iii)$]  If $X^-$ is of the type $\Fo$, doing $(\xi,V,W)=(\e^{a^-\overline{\tau}}, \e^{c^-\ot},\tan(S\ot/2))$ then we get $V=\xi^{\beta_1}, W=\tan(\ln(\xi^{\beta_2}))$ and $W=\tan(\ln(V^{\beta}))$ where $\beta_1=c^-/a^-, \beta_2=S/2a^-, S=\sqrt{-((c^-)^2+4d^-)}, \beta=S/2c^-$,
$$
W=\dfrac{2S(\wy_0-\wy_1)}{S^2\wy_0\wy_1+(c^-)^2\wy_0\wy_1+2c^-\wy_0+2c^-\wy_1+4}, \ \ V=\dfrac{S^2\wy_0^2+(c^-)^2\wy_0^2+4c^-\wy_0+4}{S^2\wy_1^2+(c^-)^2\wy_1^2+4c^-\wy_1+4}
$$
$$
\xi=\frac{(S^2+(c^-)^2+4a^-)m+((S^2+(c^-)^2)a^-+4(a^-)^2(a^--c^-))\wx_0+4b^-(a^-\wy_0+1)}{(S^2+(c^-)^2+4a^-)m+((S^2+(c^-)^2)a^-+4(a^-)^2(a^--c^-))\wx_1+4b^-(a^-\wy_1+1)}.
$$

\item[$(iv)$]  If $X^-$ is of the type $\Ce$, doing $(\xi,V,W)=(\e^{a^-\overline{\tau}},\cos(\beta\ot),\sin(\beta\ot))$ then we get  $V=\cos(\ln(\xi^{\frac{\beta}{a^-}})), W=\sin(\ln(\xi^{\frac{\beta}{a^-}}))$ and $W^2+V^2=1$, where $ \beta=\sqrt{-d^-}$ and
$$
V=\frac{-\beta^2\wy_0\wy_1-1}{-\beta^2\wy_1^2-1}, \ \ W=\frac{\beta(\wy_1-\wy_0)}{-\beta^2\wy_1^2-1},
$$
$$
\xi=\frac{\beta^2a^-\wx_0+(a^-)^3\wx_0+\beta^2m+b^-a^-\wy_0+m(a^-)^2+b^-}{\beta^2a^-\wx_1+(a^-)^3\wx_1+\beta^2m+b^-a^-\wy_1+m(a^-)^2+b^-}.
$$

\item[$(v)$] If $X^-$ is of the type $\DD_1$, doing $(\xi,V,W)=(\e^{a^-\overline{\tau}},\e^{c^-\ot},c^-\ot)$ then we get $V=\xi^{\beta_1}, W=\ln(\xi^{\beta_1})$ and $W=\ln(V)$, where $\beta_1=c^-/a^-, \beta=1/c^-$,
$$
W=c^-(\wy_0-\wy_1), \ \ V=\dfrac{c^-\wy_0+1}{c^-\wy_1+1},
$$
$$
\xi=\frac{-c^-(a^-)^2\wx_0+(a^-)^3\wx_0+b^-a^-\wy_0-c^-ma^-+m(a^-)^2+b^-}{-c^-(a^-)^2\wx_1+(a^-)^3\wx_1+b^-a^-\wy_1-c^-ma^-+m(a^-)^2+b^-}.
$$

\item[$(vi)$] If $X^-$ is of the type $\DD_2$, doing $(\xi,V,W)=(\e^{a^-\overline{\tau}},\ot,\ot^2)$ then we get $V=\ln(\xi^\beta), W=(\ln(\xi^\beta))^2$ and $W=V^2$ where $\beta=1/a^-$, 
$$
W=2\wy_1(\wy_1-\wy_0), \ \ V=\wy_0-\wy_1,
$$
$$
\xi=\frac{(a^-)^3\wx_0+b^-a^-\wy_0+m(a^-)^2+b^-}{(a^-)^3\wx_1+b^-a^-\wy_1+m(a^-)^2+b^-}.
$$

\end{enumerate}
\end{lema}
\begin{proof}
The proof is analogous to the proof of Lemma~\ref{selacentrosup}.
\end{proof}

\begin{thm}\textbf{(Kovanskii, \cite{Rolle})} \label{rolle}
Let $X$ be a $C^1$ planar vector field without singular points in an open region $\Omega\subset\mathbb{R}^2$. If a $C^1$ curve, $\gamma\subset\Omega$, intersects an integral curve of $X$ at two points then in between these points, there exists a point of tangency between $\gamma$ and $X$.
\end{thm}

\subsection{Proof of Theorem~\ref{teoprincipal}}\label{secth1}
Considering the boundary value problem~\eqref{problemafronteira1} and statement~$(i)$ of Lemma~\ref{selacentrosup} follows that $\alpha>1$, $0<v,w<1$ and $w=v^\alpha$. Expliciting  $x_0, y_0, y_1$ in \eqref{vwsa} we get  
$$
y_0=-\dfrac{(-1+\alpha)(\alpha v-vw-\alpha+v)}{c^+(vw-1)\alpha}, \ \  y_1=-\dfrac{(-1+\alpha)(\alpha vw-\alpha w-w+1)}{c^+(vw-1)\alpha},
$$ 
and $x_1=\rho x_0+B$, where
$$
B=\frac{4b^+(a^+\rho y_0-a^+y_1-\rho+1)}{a^+(4(a^+)^2-4a^+c^++(c^+)^2-s^2)}.
$$

From statement $(i)$ of Lemma~\ref{selacentroinf} and expliciting $\wx_1$ in \eqref{VWsa}, we get
\[
W=\frac{S\wy_0-c^-\wy_0-2}{S\wy_1-c^-\wy_1-2}, \ \ V=\frac{S\wy_1+c^-\wy_1+2}{S\wy_0+c^-\wy_0+2} \mbox{ and }
\wx_1=\displaystyle\frac{1}{\xi}{\wx_0}+C,
\]
where 
$$
C=\frac{4a^-b^-(\xi\wy_1-\wy_0)+(\xi-1)m(4(a^-)^2-4c^-a^--(S^2-(c^-)^2))+4(\xi-1)b^-}{a^-\xi(c^--2a^-+S)(-c^-+2a^-+S)}.
$$

Substituting $\wy_0=y_0$ and $\wy_1=y_1$ in the expressions of $V,W$, we consider in the region $\Delta=(0,1)\times(0,1)$ contained in the plane $vw$, the curves 
\begin{equation}\label{cfcF}
\begin{array}{ll}
C_f=&\{(v,w)\in \Delta;f(v,w)=w-v^{\alpha}=0\},\\ 
C_F=& \{(v,w)\in \Delta;F(v,w)=W-V^{\beta}=0\},
\end{array}
\end{equation}
where
$$
V=\frac{(\alpha^2c^-+\alpha\beta c^+-\alpha c^+-\alpha c^-)vw+(-\alpha^2c^-+c^-)w-\alpha\beta c^+ + \alpha c^+ +\alpha c^--c^-}{(\alpha\beta c^+-\alpha c^+-\alpha c^-+c^-)vw+(\alpha^2c^--c^-)v-\alpha^2c^--\alpha\beta c^+ + \alpha c^+ + \alpha c^-},
$$
$$
W=\frac{(c^-\beta(1-\alpha)+c^+\alpha(1-\beta))vw+(\alpha^2-1)c^-\beta v+\alpha(c^+(\beta-1)+c^-\beta(1-\alpha))}{(c^-\alpha\beta(\alpha-1)+c^+\alpha(1-\beta))vw+(1-\alpha^2)c^-\beta w+\alpha c^+(\beta-1)+c^-\beta(\alpha-1)}.
$$

Thus, for each point of intersection of the curves $C_f$ and $C_F$, the piecewise linear vector field $Z$ has an invariant cylinder. Note that the curve $C_F$ does not depend of $a^+$ or $a^-$. Consequently, when $a^+a^-=0$ the number of invariant cilinders is the same.

 When $v\rightarrow 1^-$ we have that $w\rightarrow 1, y_0\rightarrow0$ and $y_1\rightarrow0$. From this fact, follows that  $V\rightarrow1$, $W\rightarrow1$ and both curves $C_f$ and $C_F$ pass through the point $(1,1).$

 Let $\widetilde{X}=(v,\alpha w)$ be a vector field defined in $\Delta$. So,  $C_f$ is an integral curve of $\widetilde{X}$.
 Consider the following system
\begin{equation}\label{sis}
\{F(v,w)=0,\nabla F(v,w)\cdotp \widetilde{X}=0\}.
\end{equation}

We get $\nabla F(v,w)\cdotp \widetilde{X}=f_1(v,w)f_2(v,w)$ where
$$
f_1(v,w)=\frac{y_0\widehat{y_0}y_1\widehat{y_1}(\beta+1)(\alpha+1)\beta (c^-)^2}{\widehat{D}},
$$
$$
f_2(v,w)=\kappa w(v-1)^2+\lambda v(w-1)^2,
$$
with 
$$
\kappa=\alpha^2(\alpha c^-+\beta c^+-c^+-c^-)(\alpha\beta c^--\beta c^+-\beta c^-+c^+),
$$
$$ 
\lambda=(\alpha\beta c^+-\alpha c^+-\alpha c^-+c^-)(\alpha\beta c^++\alpha\beta c^--\alpha c^+-\beta c^-).
$$
 We denote by $\widehat{y_0}, \widehat{y_1}$ and $\widehat{D}$ the denominators of $y_0, y_1$ and $\nabla F(v,w)\cdotp \widetilde{X}$ respectively. Since that $y_0, y_1$ and $\nabla F(v,w)\cdotp \widetilde{X}$ are well defined we obtain $\widehat{y_0}\widehat{y_1}\widehat{D}\neq0$.  Note that the system~\eqref{sis} is equivalent the $\{F(v,w)=0,f_2(v,w)=0\}$.  Let $C_{f_2}$ be the curve $\{(v,w)\in \Delta;f_2(v,w)=0\}$. Now, we get that $(v,w)\in \Delta\cap C_f\cap C_{f_2}$ if and only if $(v,w)$ satisfies $\{w=v^\alpha,\kappa w(v-1)^2+\lambda v(w-1)^2=0 \}$, or equivalently
 $$
\dfrac{\lambda v(v^\alpha-1)^2}{\kappa v^\alpha(v-1)^2}+1=0.
 $$
 This equation admits one zero for $v\in(0,1)$ if $1+\alpha^2\lambda/\kappa>0$, otherwise it does not admit zeros in $(0,1)$. 
 
  \begin{figure}[!htb]
 \begin{center}
 \def\svgwidth{0.25\textwidth}
 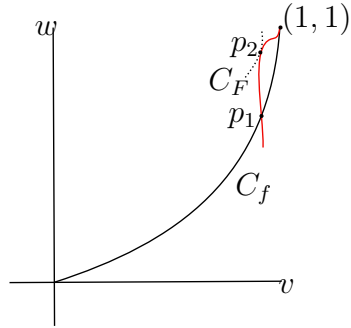
 \end{center}
 \caption{Curves $C_f$ and $C_F$.}\label{con1}
 \end{figure}

Now, we will prove the statement~$(1)$ of Theorem~\ref{teoprincipal}. So, we suppose that $\kappa^2+\lambda^2\neq0$ and $\kappa\lambda=0$. Without loss of generality we can suppose that $\kappa\neq0$ and $\lambda=0$ (the proof when $\kappa=0$ and $\lambda\neq0$ is analogous). In this case $f_2=\kappa w(v-1)^2$. Now, we assume that $C_f$ and $C_F$ intersect at a point $p_1$ in $\Delta$ (see Figure~\ref{con1}). Follows from Theorem~\ref{rolle} that there is $p_2\in\Delta$ and this point is a solution of $\{F=0,f_2=0\}$, but this is a contradiction since $f_2\neq0$ in $\Delta$. Therefore there are no invariant cylinders and consequently there are no periodic orbits, i.e., the differential system has a scroll.
 
 
 The case $\kappa^2+\lambda^2\neq0$ and $\kappa\lambda>0$ is analogous, since that in this case $f_2\neq0$ in $\Delta$.\\
 Suppose now that $1+\alpha^2\lambda/\kappa\le0$ and $\kappa\lambda<0$.
  \begin{figure}[!htb]
  \begin{center}
  \def\svgwidth{0.65\textwidth}
  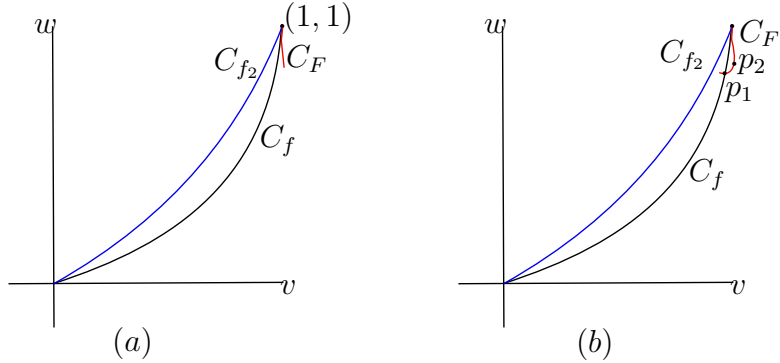
  \end{center}
  \caption{Curves $C_{f_2}$, $C_f$ and $C_F$.}\label{conf1}
  \end{figure}
The curves $C_f$, $C_F$ and $C_{f_2}$ pass through the point $(1,1)$ and we have that $\nabla F(1,1)=(F_v(1,1),F_w(1,1))$, where $F_v(1,1)>0$ and $F_w(1,1)<0$, $\widetilde{X}F(1,1)=\widetilde{X}^2F(1,1)=0$ and $\widetilde{X}^3F(1,1)>0$. We illustrate the relative positions between these curves in the  Figure~\ref{conf1} ($a$).   Assume that $C_f$ and $C_F$ intersect at a point $p_1$ in $\Delta$ (see Figure~\ref{conf1} ($b$)). Follows from Theorem~\ref{rolle} that there is $p_2\in\Delta$ and this point is a solution of $\{F=0,f_2=0\}$, i.e., $p_2\in C_{f_2}$, but this is a contradiction.

Now, in order to prove the statement~$(2)$ we assume that $\kappa\lambda<0$ and $1+\alpha^2\lambda/\kappa>0$. In this case
$\nabla F(1,1)=(F_v(1,1),F_w(1,1))$, where $F_v(1,1)>0$ and $F_w(1,1)<0$, $\widetilde{X}F(1,1)=\widetilde{X}^2F(1,1)=0$ and $\widetilde{X}^3F(1,1)<0$ (see Figure~\ref{conf222} ($a$)).

Assume that $C_f$ and $C_F$ intersect at two points $p_1$ and $p_2$ in $\Delta$ (see Figure~\ref{conf222} ($b$) ).  Follows from Theorem~\ref{rolle} that there are $p_3, p_4\in\Delta$ which  are solutions of $\{F=0,f_2=0\}$, i.e., $p_3, p_4\in C_{f_2}$, but this is a contradiction since that $C_{f_2}$ intersects $C_{f}$ at most at one point.

\begin{figure}[!htb]
 \begin{center}
   \def\svgwidth{0.65\textwidth}
   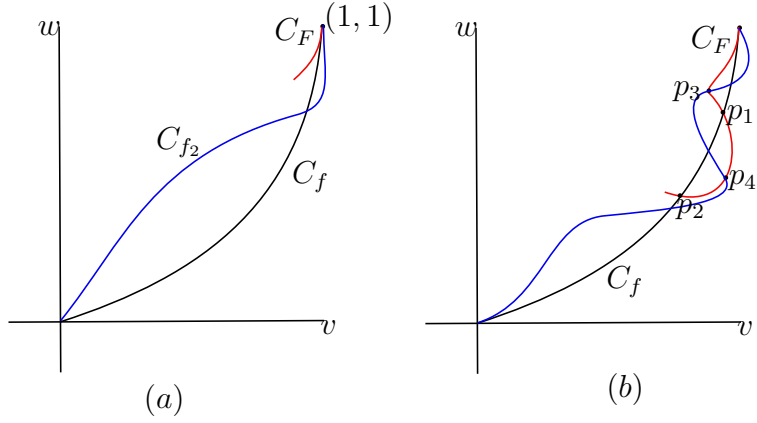
   \end{center}
   \caption{Curves $C_{f_2}$, $C_f$ and $C_F$.}\label{conf222}
   \end{figure}
If the intersection points $p_1$ or $p_2$ of $C_f\cap C_F$ are tangent with odd (or even) multiplicity the proof is analogous.

Finally, if $\kappa=\lambda=0$, then $\nabla F(v,w)\cdotp \widetilde{X}\equiv0$ and thus the curves $C_f$ and $C_F$ are coincident. In this case there is a continuous of invariant cylinders.

The proof for all the other cases is similar to this one done for the case $(\Sa,\Sa)$. In the Tables~\ref{table1} and~\ref{table3}, we present the distinguished elements necessary to analyze the number of intersections points between curves $C_f$ and $C_F$, consequently the number of invariant cylinders. Thus, we conclude the proof of Theorem~\ref{teoprincipal}.

\subsection{Proof the Theorem~\ref{teoprincipal2}}\label{secth2} We present the complete proof about existence of isolated periodic orbits when the vector field $Z$ is of type ($\Sa$,$\Sa$). So, under this hypotesis, we start observing that if $Z$ has no invariant cylinder then do not exist limit cycles.  As the curve $C_F$ given in \eqref{cfcF} does not depend of $a^+$ or $a^-$, the number of invariant cylinders remains the same, independently of the configuration of $a^+$ and $a^-$.

Suppose that $(a^+)^2+(a^-)^2=0$. From boundary value problems \eqref{problemafronteira1} and \eqref{problemafronteira2} and from \eqref{retas} of the Proposition~\ref{sollinear},  we obtain  respectively $x_1=x_0+\eta y_0+\mu$ and $\wx_1=\wx_0+\widetilde{\eta}\wy_1+\widetilde{\mu}$. Here, we obtain $(\eta,\mu)$ and $(\widetilde{\eta},\widetilde{\mu})$ are obtained directly from Proposition~\ref{sollinear} replacing $(\gamma,\delta,M,\sigma,\psi,\pm1)=(a^+,b^+,0,c^+,d^+,-1)$ and $(\gamma,\delta,M,\sigma,\psi,\pm1)=(a^-,b^-,m,c^-,d^-,1)$, respectively. Fixing an invariant cylinder $\wy_0=y_0$ and $\wy_1=y_1$, we get that the return times $\tau$ and $\overline{\tau}$ are also fixed. Thus, in this invariant cylinder, the number of limit cycles is given by the intersections of the straight lines $r^\pm$ given by
$$
r^+: x_1= x_0+\eta y_0+\mu \ \ \ \mbox{ with } \ \ \ r^-: \wx_1=\wx_0-(\widetilde{\eta}y_1+\widetilde{\mu}),
$$
where $\wx_0=x_0$ and $\wx_1=x_1$. Doing $(\wB,\wC)=(\eta y_0+\mu,-(\widetilde{\eta}y_1+\widetilde{\mu}))$ we obtain that either all solutions are closed in the cylinder, if  $\wB=\wC$, or there is no closed solutions in this cylinder when $\wB\neq \wC$.

Suppose that $(a^+)^2+(a^-)^2\neq0$ and $Z$ at most one invariant cylinder. Fixed the invariant cylinder, the number of limit cycles is given by the intersections of $x_1=\rho x_0+B$ if $a^+\neq0$ (or $x_1=x_0+\wB$ if $a^+=0$) and $x_1=x_0/\xi +C$ if $a^-\neq0$ (or $x_1=x_0+\wC$ if $a^-=0$), where $\rho=\e^{a^+\tau}$, $\xi=\e^{a^-\overline{\tau}}$, $(\wB,\wC)$ are obtained as above and  $(B,C)$ obtained in the proof of Theorem~\ref{teoprincipal}. Thus, there is at most one limit cycle.

Suppose that  $(a^+)^2+(a^-)^2\neq0, a^+a^-\geq0$ and $Z$ has infinitely many invariant cylinders. Suppose initially  $a^+a^->0$.  We will show that in each invariant cylinder there is a unique isolated periodic orbit. Indeed, in each cylinder, the orbit periodic is given by the intersection of the straight lines $r^\pm$ given by
$$
r^+:x_1=\rho x_0+B \ \ \ \mbox{and} \ \ \ r^-:x_1=\frac{1}{\xi}x_0+C.
$$
These straight lines has a unique intersection point provided that $\rho\neq 1/\xi$. Note that $\rho=1/\xi \Leftrightarrow\tau=-a^-\overline{\tau}/a^+$, where $\tau$ and $\overline{\tau}$  are positives. With the hypothesis of that $a^+a^->0$, the relation $\tau=-a^-\overline{\tau}/a^+$ can not be satisfied and thus $\rho\neq1/\xi$. The intersection point in each cylinder is given by
$$
x_0=\frac{C-B}{\frac{1}{\xi}-\rho}\ \ \mbox{e}\ \ x_1=\rho\frac{C-B}{\frac{1}{\xi}-\rho}+B.
$$
Varying continuously the cylinders, the terms $x_0$ and $x_1$ also range continuously, and we obtain one invariant surface formed of periodc orbits, where each orbit is an invariant cylinder. If $a^+=0$ and $a^-\neq0$, the periodic orbit in each cylinder is given by intersection of stright lines
$$
x_1=x_0+\wB \ \ \ \mbox{and} \ \ \ x_1=\frac{1}{\xi}x_0+C.
$$
The case $a^+\neq0$ and $a^-=0$ follows analogously. 

\begin{ex} Consider that $a^+=\frac{1}{20},b^+=0,c^+=-\frac{7}{16}, d^+=\frac{5}{8},c^-=\frac{1}{2},d^-=\frac{3}{16},a^-=m=b^-=1$. 
\begin{figure}[!htb] 
\begin{center}
\def\svgwidth{0.4\textwidth}
           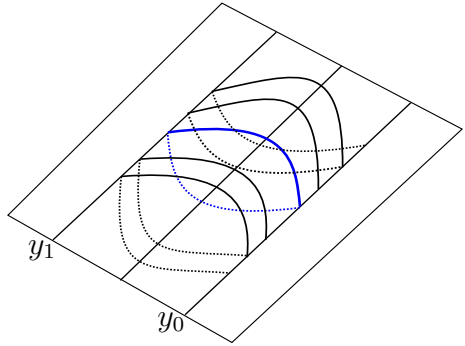
\end{center}
\caption{Invariant cylinder with a unique limit cycle.}\label{1limitcycle}

\end{figure}
\end{ex}

\smallskip

\begin{ex} Consider that  $c^+=c^-=m=0, a^+=b^-=d^+=-1, d^-=-2, b^+=1$ and $a^-=-2$. In this case we get that the vector field $Z$ has infinitely many invariant cylinders. In each invariant cylinder there is a unique limit cycle. All limit cycles form a cone. (see Figure~\ref{cone}).
\end{ex}
\begin{figure}[!htb]
\begin{center}
\def\svgwidth{0.45\textwidth}
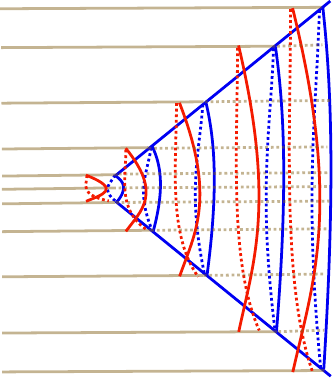
\end{center}
\caption{Cone of Periodic Orbits.}\label{cone}
\end{figure}

\subsection{The case focus-focus}\label{secth3}

As mentioned earlier, the case (Fo,Fo) is not resolved with the techniques of Subsection~\ref{secth1}. We use here \cite{Ja}. For this, we get a different canonical form of~\eqref{formanormal}. Proceeding as in Proposition~\ref{propformacanonica}, doing the twin change of variables given by\\
$\varphi^{+}(x,y,z)=(x-(b_1^+/b_2^+)y+((b_1^+a_{11}^+-b_1^+a_{22}^++a_{12}^+b_2^+)/b_2^+)z,y+a_{22}^+z,z)$ on $\Sigma^+$, \\
$\varphi^{-}(x,y,z)=(x-(b_1^+/b_2^+)y+((b_2^+a_{12}^--b_1^+a_{22}^-+b_1^+a_{11}^-)/b_2^+)z,y+a_{22}^-z,z)$ on $\Sigma^-$,\\
in vector field
 $$
 \begin{array}{l}
 X^+(x,y,z)=(a_{11}^+x+ a_{12}^+y+a_{13}^+z+b_1^+,a_{22}^+y+a_{23}^+z+b_2^+,-y+a_{33}^+z),\\
 X^-(x,y,z)=(a_{11}^-x+ a_{12}^-y+a_{13}^-z+b_1^-,a_{22}^-y+a_{23}^-z+b_2^-,-y+a_{33}^-z),
 \end{array}
$$
we obtain 
\begin{equation}\label{formanormalfoco}
 \begin{array}{l}
X^+(x,y,z)=(a^{+}x+b^{+}z,D_2z+a_2,-y+T_2z),\\
X^-(x,y,z)=(a^{-}x+b^{-}z+m,D_1z+a_1,-y+T_1z).
 \end{array}
\end{equation}

To study the case (Fo-Fo), we must assume that $a_2>0, a_1<0, T_{i}^2-4D_i<0$ and $T_i\neq0$, with $i=1,2$.
The Corollary~\ref{cilindros} is also satisfied in this case. Thus, to determine quotas for the number of limit cycles, we will determine first the number of cylinders invariants. We note that in plane $(y,z)$, the vectors field $X^+$ and $X^-$ are as in Theorem 4 of \cite{Ja}. As the cylinders are determined by straight lines of type $r_0=\{(x,y,z)\in\mathbb{R}^3:y=y_0,z=0\}$  and $r_1=\{(x,y,z)\in\mathbb{R}^3:y=y_1,z=0\}$, we will consider only the components in the plane $(y,z)$.
By Theorems 4 and 5 of \cite{Ja}, there is at most one limit cycle for
$$
 \begin{array}{lll}
y'=D_2z+a_2,\\
z'=-y+T_2z,
 \end{array} \ \ \mbox{if $z>0$},\ \ \ \ \  \begin{array}{lll}
 y'=D_1z+a_1,\\
 z'=-y+T_1z,
  \end{array} \ \ \mbox{if $z<0$},
  $$
with $a_2>0, a_1<0, T_{i}^2-4D_i<0$ and $T_i\neq0$, with $i=1,2$.

 Reflecting this result to our case, the limit cycle becomes an invariant cylinder. Thus, for $X^+$ and $X^-$ as in~\eqref{formanormalfoco}, there is at most one invariant cylinder. Let us consider the boundary value problems
 $$
  \begin{array}{lll}
 x'=a^{+}x+b^{+}z,\\
 y'=D_2z+a_2,\\
 z'=-y+T_2z,
  \end{array}\ \ \ \ \  \begin{array}{lll}
  (x(0),y(0),z(0))=(x_0,y_0,0),\\
  (x(\tau),y(\tau),z(\tau))=(x_1,y_1,0),
  \end{array}
$$
and
$$
 \begin{array}{lll}
x'=a^{-}x+b^{-}z+m,\\
y'=D_1z+a_1,\\
z'=-y+T_1z,
 \end{array}\ \ \ \ \  \begin{array}{lll}
 (x(0),y(0),z(0))=(\wx_1,\wy_1,0),\\
 (x(\overline{\tau}),y(\overline{\tau}),z(\overline{\tau}))=(\wx_0,\wy_0,0).
 \end{array}
$$
Fixed the invariant cylinder, $y_0=\wy_0$ and $y_1=\wy_1$ the times $\tau$ and $\overline{\tau}$ are also fixed. With the boundary value problems above, we can write again
$$
x_1=\rho x_0+\overline{\overline{B}}, \ \ \ \mbox{and}\ \ \ \wx_1=\frac{1}{\xi}\wx_0+\overline{\overline{C}},
$$
where $\rho=\e^{a^{+}\tau}$, $\xi=\e^{a^{-}\overline{\tau}}$ with $\overline{\overline{B}}$ and $\overline{\overline{C}}$
obtained as Proposition~\ref{sollinear}. 
Thus, doing $x_0=\wx_0$ and $x_1=\wx_1$, exists in this cylinder at most one limit cycle. The cases where $a^{+}a^{-}=0$ are studied analogously, noting once again that if $a^{+}=0$ we obtain $x_1=x_0+\overline{\overline{B}}$ and if $a^{-}=0$, $\wx_1=\wx_0+\overline{\overline{C}}$.

 \section*{Acknowledgments}
Both authors are partially supported by the FA\-PEG, by the CNPq
grants numbers 475623/2013-4 and 306615/2012-6 and, by the CAPES
grant numbers PRO\-CAD\- 88881.068462/2014-01 and by CSF/PVE-88881.\-030454/2013-01.

\medskip
\bibliographystyle{plain}
\bibliography{sewing}

\end{document}

%% file: 3cilindros3.pdf_tex
\begingroup%
  \makeatletter%
  \providecommand\color[2][]{%
    \errmessage{(Inkscape) Color is used for the text in Inkscape, but the package 'color.sty' is not loaded}%
    \renewcommand\color[2][]{}%
  }%
  \providecommand\transparent[1]{%
    \errmessage{(Inkscape) Transparency is used (non-zero) for the text in Inkscape, but the package 'transparent.sty' is not loaded}%
    \renewcommand\transparent[1]{}%
  }%
  \providecommand\rotatebox[2]{#2}%
  \ifx\svgwidth\undefined%
    \setlength{\unitlength}{353.55054134bp}%
    \ifx\svgscale\undefined%
      \relax%
    \else%
      \setlength{\unitlength}{\unitlength * \real{\svgscale}}%
    \fi%
  \else%
    \setlength{\unitlength}{\svgwidth}%
  \fi%
  \global\let\svgwidth\undefined%
  \global\let\svgscale\undefined%
  \makeatother%
  \begin{picture}(1,0.52600751)%
    \put(0,0){\includegraphics[width=\unitlength]{3cilindros3.pdf}}%
    \put(0.39417012,0.00383957){\color[rgb]{0,0,0}\makebox(0,0)[lb]{\smash{$(b)$}}}%
    \put(0.72036648,0.00640242){\color[rgb]{0,0,0}\makebox(0,0)[lb]{\smash{$(c)$}}}%
    \put(0.09673585,0.00659865){\color[rgb]{0,0,0}\makebox(0,0)[lb]{\smash{$(a)$}}}%
  \end{picture}%
\endgroup%

%% file: calhasupinf1.pdf_tex
\begingroup%
  \makeatletter%
  \providecommand\color[2][]{%
    \errmessage{(Inkscape) Color is used for the text in Inkscape, but the package 'color.sty' is not loaded}%
    \renewcommand\color[2][]{}%
  }%
  \providecommand\transparent[1]{%
    \errmessage{(Inkscape) Transparency is used (non-zero) for the text in Inkscape, but the package 'transparent.sty' is not loaded}%
    \renewcommand\transparent[1]{}%
  }%
  \providecommand\rotatebox[2]{#2}%
  \ifx\svgwidth\undefined%
    \setlength{\unitlength}{320.37992671bp}%
    \ifx\svgscale\undefined%
      \relax%
    \else%
      \setlength{\unitlength}{\unitlength * \real{\svgscale}}%
    \fi%
  \else%
    \setlength{\unitlength}{\svgwidth}%
  \fi%
  \global\let\svgwidth\undefined%
  \global\let\svgscale\undefined%
  \makeatother%
  \begin{picture}(1,0.45170771)%
    \put(0,0){\includegraphics[width=\unitlength]{calhasupinf1.pdf}}%
    \put(0.21560994,0.07355849){\color[rgb]{0,0,0}\makebox(0,0)[lb]{\smash{$y_0$}}}%
    \put(0.06080654,0.12497674){\color[rgb]{0,0,0}\makebox(0,0)[lb]{\smash{$y_1$}}}%
    \put(0.7706926,0.00529637){\color[rgb]{0,0,0}\makebox(0,0)[lb]{\smash{$y_0$}}}%
    \put(0.61462806,0.05671462){\color[rgb]{0,0,0}\makebox(0,0)[lb]{\smash{$y_1$}}}%
    \put(0.38123838,0.40312265){\color[rgb]{0,0,0}\makebox(0,0)[lb]{\smash{$X^+$}}}%
    \put(0.96643218,0.35267488){\color[rgb]{0,0,0}\makebox(0,0)[lb]{\smash{$X^-$}}}%
  \end{picture}%
\endgroup%

%% file: con1.pdf_tex
\begingroup%
  \makeatletter%
  \providecommand\color[2][]{%
    \errmessage{(Inkscape) Color is used for the text in Inkscape, but the package 'color.sty' is not loaded}%
    \renewcommand\color[2][]{}%
  }%
  \providecommand\transparent[1]{%
    \errmessage{(Inkscape) Transparency is used (non-zero) for the text in Inkscape, but the package 'transparent.sty' is not loaded}%
    \renewcommand\transparent[1]{}%
  }%
  \providecommand\rotatebox[2]{#2}%
  \ifx\svgwidth\undefined%
    \setlength{\unitlength}{166.23833202bp}%
    \ifx\svgscale\undefined%
      \relax%
    \else%
      \setlength{\unitlength}{\unitlength * \real{\svgscale}}%
    \fi%
  \else%
    \setlength{\unitlength}{\svgwidth}%
  \fi%
  \global\let\svgwidth\undefined%
  \global\let\svgscale\undefined%
  \makeatother%
  \begin{picture}(1,1.08745528)%
    \put(0,0){\includegraphics[width=\unitlength]{con1.pdf}}%
    \put(0.76797457,0.7199616){\color[rgb]{0,0,0}\makebox(0,0)[lb]{\smash{$p_1$}}}%
    \put(0.77623823,0.9727102){\color[rgb]{0,0,0}\makebox(0,0)[lb]{\smash{$p_2$}}}%
    \put(0.94381158,0.11765557){\color[rgb]{0,0,0}\makebox(0,0)[lb]{\smash{$v$}}}%
    \put(0.08909352,1.02962306){\color[rgb]{0,0,0}\makebox(0,0)[lb]{\smash{$w$}}}%
    \put(0.95545123,1.05675966){\color[rgb]{0,0,0}\makebox(0,0)[lb]{\smash{$(1,1)$}}}%
    \put(0.69855125,0.8340577){\color[rgb]{0,0,0}\makebox(0,0)[lb]{\smash{$C_{F}$}}}%
    \put(0.79296628,0.46288587){\color[rgb]{0,0,0}\makebox(0,0)[lb]{\smash{$C_{f}$}}}%
  \end{picture}%
\endgroup%

%% file: conf111.pdf_tex
\begingroup%
  \makeatletter%
  \providecommand\color[2][]{%
    \errmessage{(Inkscape) Color is used for the text in Inkscape, but the package 'color.sty' is not loaded}%
    \renewcommand\color[2][]{}%
  }%
  \providecommand\transparent[1]{%
    \errmessage{(Inkscape) Transparency is used (non-zero) for the text in Inkscape, but the package 'transparent.sty' is not loaded}%
    \renewcommand\transparent[1]{}%
  }%
  \providecommand\rotatebox[2]{#2}%
  \ifx\svgwidth\undefined%
    \setlength{\unitlength}{427.95440583bp}%
    \ifx\svgscale\undefined%
      \relax%
    \else%
      \setlength{\unitlength}{\unitlength * \real{\svgscale}}%
    \fi%
  \else%
    \setlength{\unitlength}{\svgwidth}%
  \fi%
  \global\let\svgwidth\undefined%
  \global\let\svgscale\undefined%
  \makeatother%
  \begin{picture}(1,0.45879233)%
    \put(0,0){\includegraphics[width=\unitlength]{conf111.pdf}}%
    \put(0.97375481,0.08210397){\color[rgb]{0,0,0}\makebox(0,0)[lb]{\smash{$v$}}}%
    \put(0.64174073,0.43635657){\color[rgb]{0,0,0}\makebox(0,0)[lb]{\smash{$w$}}}%
    \put(0.87877205,0.39772749){\color[rgb]{0,0,0}\makebox(0,0)[lb]{\smash{$C_{f_2}$}}}%
    \put(0.91798284,0.23492799){\color[rgb]{0,0,0}\makebox(0,0)[lb]{\smash{$C_{f}$}}}%
    \put(0.98603999,0.42507167){\color[rgb]{0,0,0}\makebox(0,0)[lb]{\smash{$C_{F}$}}}%
    \put(0.96756534,0.35109159){\color[rgb]{0,0,0}\makebox(0,0)[lb]{\smash{$p_1$}}}%
    \put(0.9847707,0.38887895){\color[rgb]{0,0,0}\makebox(0,0)[lb]{\smash{$p_2$}}}%
    \put(0.36660933,0.08207491){\color[rgb]{0,0,0}\makebox(0,0)[lb]{\smash{$v$}}}%
    \put(0.03459523,0.43632748){\color[rgb]{0,0,0}\makebox(0,0)[lb]{\smash{$w$}}}%
    \put(0.37113073,0.44686866){\color[rgb]{0,0,0}\makebox(0,0)[lb]{\smash{$(1,1)$}}}%
    \put(0.27729171,0.38636839){\color[rgb]{0,0,0}\makebox(0,0)[lb]{\smash{$C_{f_2}$}}}%
    \put(0.33916241,0.28116315){\color[rgb]{0,0,0}\makebox(0,0)[lb]{\smash{$C_{f}$}}}%
    \put(0.3748888,0.39032594){\color[rgb]{0,0,0}\makebox(0,0)[lb]{\smash{$C_{F}$}}}%
    \put(0.13991698,0.00388348){\color[rgb]{0,0,0}\makebox(0,0)[lb]{\smash{$(a)$}}}%
    \put(0.76596824,0.00320993){\color[rgb]{0,0,0}\makebox(0,0)[lb]{\smash{$(b)$}}}%
  \end{picture}%
\endgroup%

%% file: conf222.pdf_tex
\begingroup%
  \makeatletter%
  \providecommand\color[2][]{%
    \errmessage{(Inkscape) Color is used for the text in Inkscape, but the package 'color.sty' is not loaded}%
    \renewcommand\color[2][]{}%
  }%
  \providecommand\transparent[1]{%
    \errmessage{(Inkscape) Transparency is used (non-zero) for the text in Inkscape, but the package 'transparent.sty' is not loaded}%
    \renewcommand\transparent[1]{}%
  }%
  \providecommand\rotatebox[2]{#2}%
  \ifx\svgwidth\undefined%
    \setlength{\unitlength}{372.75832201bp}%
    \ifx\svgscale\undefined%
      \relax%
    \else%
      \setlength{\unitlength}{\unitlength * \real{\svgscale}}%
    \fi%
  \else%
    \setlength{\unitlength}{\svgwidth}%
  \fi%
  \global\let\svgwidth\undefined%
  \global\let\svgscale\undefined%
  \makeatother%
  \begin{picture}(1,0.53665121)%
    \put(0,0){\includegraphics[width=\unitlength]{conf222.pdf}}%
    \put(0.98342851,0.10218729){\color[rgb]{0,0,0}\makebox(0,0)[lb]{\smash{$v$}}}%
    \put(0.60225146,0.50889576){\color[rgb]{0,0,0}\makebox(0,0)[lb]{\smash{$w$}}}%
    \put(0.92084955,0.49293772){\color[rgb]{0,0,0}\makebox(0,0)[lb]{\smash{$C_{F}$}}}%
    \put(0.96942222,0.39680815){\color[rgb]{0,0,0}\makebox(0,0)[lb]{\smash{$p_1$}}}%
    \put(0.90167381,0.26835696){\color[rgb]{0,0,0}\makebox(0,0)[lb]{\smash{$p_2$}}}%
    \put(0.89733783,0.42824345){\color[rgb]{0,0,0}\makebox(0,0)[lb]{\smash{$p_3$}}}%
    \put(0.97213217,0.3079221){\color[rgb]{0,0,0}\makebox(0,0)[lb]{\smash{$p_4$}}}%
    \put(0.8070203,0.1661791){\color[rgb]{0,0,0}\makebox(0,0)[lb]{\smash{$C_{f}$}}}%
    \put(0.42088222,0.10415145){\color[rgb]{0,0,0}\makebox(0,0)[lb]{\smash{$v$}}}%
    \put(0.03970521,0.51085989){\color[rgb]{0,0,0}\makebox(0,0)[lb]{\smash{$w$}}}%
    \put(0.42607313,0.52296194){\color[rgb]{0,0,0}\makebox(0,0)[lb]{\smash{$(1,1)$}}}%
    \put(0.19910143,0.35594531){\color[rgb]{0,0,0}\makebox(0,0)[lb]{\smash{$C_{f_2}$}}}%
    \put(0.38286718,0.30887216){\color[rgb]{0,0,0}\makebox(0,0)[lb]{\smash{$C_{f}$}}}%
    \put(0.35884522,0.50357368){\color[rgb]{0,0,0}\makebox(0,0)[lb]{\smash{$C_{F}$}}}%
    \put(0.18265465,0.00368524){\color[rgb]{0,0,0}\makebox(0,0)[lb]{\smash{$(a)$}}}%
    \put(0.80688899,0.01871747){\color[rgb]{0,0,0}\makebox(0,0)[lb]{\smash{$(b)$}}}%
  \end{picture}%
\endgroup%

%% file: descomciclo2.pdf_tex
\begingroup%
  \makeatletter%
  \providecommand\color[2][]{%
    \errmessage{(Inkscape) Color is used for the text in Inkscape, but the package 'color.sty' is not loaded}%
    \renewcommand\color[2][]{}%
  }%
  \providecommand\transparent[1]{%
    \errmessage{(Inkscape) Transparency is used (non-zero) for the text in Inkscape, but the package 'transparent.sty' is not loaded}%
    \renewcommand\transparent[1]{}%
  }%
  \providecommand\rotatebox[2]{#2}%
  \ifx\svgwidth\undefined%
    \setlength{\unitlength}{203.88539003bp}%
    \ifx\svgscale\undefined%
      \relax%
    \else%
      \setlength{\unitlength}{\unitlength * \real{\svgscale}}%
    \fi%
  \else%
    \setlength{\unitlength}{\svgwidth}%
  \fi%
  \global\let\svgwidth\undefined%
  \global\let\svgscale\undefined%
  \makeatother%
  \begin{picture}(1,0.75284964)%
    \put(0,0){\includegraphics[width=\unitlength]{descomciclo2.pdf}}%
    \put(0.04625283,0.19792793){\color[rgb]{0,0,0}\makebox(0,0)[lb]{\smash{$y_1$}}}%
    \put(0.33029958,0.03716129){\color[rgb]{0,0,0}\makebox(0,0)[lb]{\smash{$y_0$}}}%
  \end{picture}%
\endgroup%

%% file: conebbbb.pdf_tex
\begingroup%
  \makeatletter%
  \providecommand\color[2][]{%
    \errmessage{(Inkscape) Color is used for the text in Inkscape, but the package 'color.sty' is not loaded}%
    \renewcommand\color[2][]{}%
  }%
  \providecommand\transparent[1]{%
    \errmessage{(Inkscape) Transparency is used (non-zero) for the text in Inkscape, but the package 'transparent.sty' is not loaded}%
    \renewcommand\transparent[1]{}%
  }%
  \providecommand\rotatebox[2]{#2}%
  \ifx\svgwidth\undefined%
    \setlength{\unitlength}{95.72426065bp}%
    \ifx\svgscale\undefined%
      \relax%
    \else%
      \setlength{\unitlength}{\unitlength * \real{\svgscale}}%
    \fi%
  \else%
    \setlength{\unitlength}{\svgwidth}%
  \fi%
  \global\let\svgwidth\undefined%
  \global\let\svgscale\undefined%
  \makeatother%
  \begin{picture}(1,1.13503097)%
    \put(0,0){\includegraphics[width=\unitlength]{conebbbb.pdf}}%
  \end{picture}%
\endgroup%